\documentclass[11pt, a4paper,reqno]{amsart}
\usepackage[T1]{fontenc}
\usepackage{amsaddr}
\usepackage{yhmath,mathrsfs,amsthm,amsmath,amssymb,amsfonts, enumerate,lipsum, appendix,mathtools, bm}
\usepackage{bbold}
\usepackage[mathscr]{euscript}
\usepackage[normalem]{ulem} 
\usepackage{graphicx}
\usepackage{pgf,tikz}
\usepackage{tkz-euclide}
\usepackage{graphicx}
\usepackage{subcaption}
\usetikzlibrary{shapes.geometric}
\usetikzlibrary{arrows,hobby}
\usepackage{enumitem}
\usepackage[english]{babel}
\DeclareMathOperator\supp{supp}

\allowdisplaybreaks

\usepackage[sort,numbers]{natbib}
\usepackage{hyperref}
\hypersetup{
	colorlinks=true,
	linkcolor=blue,
	filecolor=magenta,      
	urlcolor=red,
	citecolor=red,
}
\usepackage[margin=0.90in]{geometry}

\allowdisplaybreaks
\usepackage{orcidlink}
\newtheorem{theorem}{Theorem}[section]

\newtheorem{proposition}[theorem]{Proposition}
\newtheorem{definition}[theorem]{Definition}
\theoremstyle{definition}

\newtheorem{remark}[theorem]{Remark}
\numberwithin{equation}{section}
\usepackage[hyperpageref]{backref}

\newcommand*\rd{\mathbb{R}^d}

\newcommand{\al} {\alpha}

\newcommand{\Om} {\Omega}
\newcommand{\la} {\lambda}

\newcommand{\no} {\nonumber}
\newcommand{\noi} {\noindent}

\newcommand{\ra} {\rightarrow}

\def\dx{\,{\rm d}x}
\def\dy{\,{\rm d}y}

\def\C{{\mathcal C}}

\def\R{{\mathbb R}}

\def\({{\Big(}}
\def\){{\Big)}}

\def\dx{\,{\rm d}x}

\DeclarePairedDelimiter\abs{\lvert}{\rvert}%
\DeclarePairedDelimiter\norm{\lVert}{\rVert}%
\def\wps{{W_0^{s,p}(\Omega )}}

\title[Fractional monotonicity over annuli]{Strict monotonicity of the first $q$-eigenvalue of the fractional $p$-Laplace operator over annuli}
\author[K. Ashok~Kumar and N. Biswas]{K Ashok Kumar\,\orcidlink{0000-0002-4098-3705} \and Nirjan Biswas$^\dagger$\,\orcidlink{0000-0002-3528-8388}}
\address{\rm 
TIFR Centre For Applicable Mathematics Bengaluru\\ Post Bag No 6503, Sharada Nagar, Bengaluru 560065, India}
\email[Ashok]{ashok23@tifrbng.res.in, srasoku@gmail.com} \email[Nirjan]{nirjan22@tifrbng.res.in}
\thanks{$^\dagger$Corresponding author}
\subjclass[2020]{Primary 35R11, 49Q10; Secondary 35B51, 35B06, 47J10}
\keywords{nonlocal operator, polarizations, strong comparison principle, Faber-Krahn inequality, monotonicity of the first eigenvalue.}
\begin{document}
\begin{abstract}
    Let $B, B'\subset \mathbb{R}^d$ with $d\geq 2$ be two balls such that $B'\subset \subset B$ and the position of $B'$ is varied within $B$. For $p\in (1, \infty ),$ $s\in (0,1)$, and $q \in [1, p^*_s)$ with $p^*_s=\frac{dp}{d-sp}$ if $sp < d$ and $p^*_s=\infty $ if $sp \geq d$, let $\lambda ^s_{p,q}(B\setminus \overline{B'})$ be the first $q$-eigenvalue of the fractional $p$-Laplace operator $(-\Delta _p)^s$ in $B\setminus \overline{B'}$ with the homogeneous nonlocal Dirichlet boundary conditions. We prove that $\lambda ^s_{p,q}(B\setminus \overline{B'})$ strictly decreases as the inner ball $B'$ moves towards the outer boundary $\partial B$. To obtain this strict monotonicity, we establish a strict Faber-Krahn type inequality for $\lambda _{p,q}^s(\cdot )$ under polarization. This extends some monotonicity results obtained by Djitte-Fall-Weth (Calc. Var. Partial Differential Equations, 60:231, 2021) in the case of $(-\Delta )^s$ and $q=1, 2$ to $(-\Delta _p)^s$ and $q\in [1, p^*_s).$ Additionally, we provide the strict monotonicity results for the general domains that are difference of Steiner symmetric or foliated Schwarz symmetric sets in $\mathbb{R}^d$.
\end{abstract}
\maketitle
\section{Introduction}
For $d\geq 1$, $p\in (1,\infty )$, and $s\in (0,1)$, the Gagliardo seminorm $[\cdot ]_{s,p}$ in $\mathbb{R}^d$ is given by
\begin{align*}
    [u]_{s,p}= \left( \int _{\mathbb{R}^d} \int _{\mathbb{R}^d} \frac{|u(x)-u(y)|^p}{|x-y|^{d+sp}}\dy \dx \right)^{\frac{1}{p}},
\end{align*}
and the fractional critical Sobolev exponent $p^*_s$ is defined as $p^*_s =\frac{dp}{d-sp}$ if $sp<d$, and  $p^*_s =\infty$ if $sp \geq d$. For a bounded open set $\Omega \subset \mathbb{R}^d$ and $q\in [1,p^*_s)$, we consider the best constant  $\lambda ^s_{p,q}(\Omega )$ in the following family of Sobolev inequalities:
\begin{equation}\label{SObolev-ineq}
    \left(\lambda ^s_{p,q}(\Omega ) \right)^{\frac{1}{p}}\|u\|_{L^q(\Omega )} \leq [u]_{s,p}, \quad \forall \, u\in W^{s,p}_0(\Omega ),
\end{equation}
where $W^{s,p}_0(\Omega ) := \{ u \in L^p(\rd) : [u]_{s,p} < \infty, \text{ and } u = 0 \text{ in } \rd \setminus \Omega  \}$ is the fractional Sobolev space. The variational characterization for $\lambda ^s_{p,q}(\Omega )$ in~\eqref{SObolev-ineq} is given by
\begin{equation}\label{var-char-Ineq}
    \lambda ^s_{p,q}(\Omega ) := \inf \left\{[u]^p_{s,p}: u\in W^{s,p}_0(\Omega ) \text{ with } \|u\|_{L^q(\Omega )}=1\right\}.
\end{equation}
For $1\leq q<p^*_s$, using the compact embedding $W^{s,p}_0(\Omega )\hookrightarrow L^q(\Omega )$ (see~\cite[Corollary~2.8]{BrLiPa}) and direct variational methods, a minimizer for~\eqref{var-char-Ineq} exists.
Without loss of generality, we can assume that the minimizer is non-negative since $[\abs{u}]_{s,p} \leq [u]_{s,p}$ holds for $u \in \wps$.
Further, any minimizer $u\in W^{s,p}_0(\Omega )$ of~\eqref{var-char-Ineq} satisfies the following identity: for any $\phi \in W^{s,p}_0(\Omega )$,
\begin{align*}
        \int _{\rd} \int _{\rd} \frac{|u(x)-u(y)|^{p-2}(u(x)-u(y))}{\abs{x-y}^{d+sp}} (\phi(x)-\phi(y))\, \dy \dx = \lambda ^s_{p,q}(\Omega )\int_{\Omega } u(x)^{q-1} \phi(x) \, \dx ,
\end{align*}
and this implies that $u$ is a weak solution to the following nonlinear fractional eigenvalue problem: 
\begin{equation}\label{Eigen_eqn}
    \begin{aligned}
        (-\Delta _p)^s u &= \lambda ^s_{p,q}(\Omega ) u^{q-1}  \mbox{ in } \Omega ,\\
        u&=0 \mbox{ in } \mathbb{R}^d\setminus \Omega \text{ with } \|u\|_{L^q(\Omega )}=1,
    \end{aligned}
\end{equation}
where $(-\Delta _p)^s$ is the fractional $p$-Laplace operator that is defined as
\begin{equation*}
    (-\Delta _p)^s u (x) = \textrm{P.V.}\int _{\mathbb{R}^d} \frac{|u(x)-u(y)|^{p-2}(u(x)-u(y))}{|x-y|^{d+sp}} \,\mathrm{d}y
\end{equation*}
where P.V. denotes the principle value. We call the best constant $\lambda _{p,q}^s(\Omega )$ to be the first \emph{$q$-eigenvalue} of the fractional $p$-Laplace operator, and the corresponding minimizer to be a first~\emph{$q$-eigenfunction} of~\eqref{Eigen_eqn}. 
In Proposition~\ref{regular}, we show that the first $q$-eigenfunctions are in $\C^{0,\alpha }(\overline{\Omega })$, for some  $\al \in (0, s]$.  Further, by the strong maximum principle~\cite[Theorem A.1]{BF2014}, the first $q$-eigenfunctions are strictly positive in $\Omega $. In the homogeneous case of $p=q$, $\lambda _{p,p}^s(\Omega )$ is the first eigenvalue of $(-\Delta _p)^s$ in $\Omega $ with the homogeneous nonlocal Dirichlet boundary condition. In the case of $q=1$, the first $1$-eigenvalue is related to the fractional $p$-torsional rigidity $\mathscr{T}_{s,p}(\Omega )$ of $\Omega$, which is defined as   
\begin{equation*}
    \mathscr{T}_{s,p}(\Omega )= \left(\int _{\Omega } w \dx \right)^{p-1}=\left( \int _{\mathbb{R}^d} \int _{\mathbb{R}^d} \frac{|w(x)-w(y)|^p}{|x-y|^{d+sp}}\dy \dx \right)^{p-1}, 
\end{equation*}
where $w\in W^{s,p}_0(\Omega )$ is a positive weak solution of the following fractional $p$-torsion problem:
\begin{equation*}
    \begin{aligned}
        (-\Delta _p)^s w &= 1 \mbox{ in } \Omega ,\\
        w&=0 \mbox{ in } \mathbb{R}^d\setminus \Omega .
    \end{aligned}
\end{equation*}
Using simple scaling arguments, one can verify that 
\begin{align}\label{relation}
    \mathscr{T}_{s,p}(\Omega )= \left(\lambda ^s_{p,1}(\Omega )\right)^{-1}.
\end{align}

In this paper, we consider the following family of annular domains in $\mathbb{R}^d$: for $0<r<R<\infty $, 
\begin{equation*}
    \mathcal{A}_{r,R}:=\left\{\Omega _t:= B_R(0)\setminus \overline{B}_r(t e_1) \subset \mathbb{R}^d : 0\leq t <R-r \right\},
\end{equation*}
where $B_r(z)$ is the open ball with radius $r\geq 0$ centered at $z\in \mathbb{R}^d$, and $\overline{B}_r(z)$ is the closure of the open ball $B_r(z)$ in $\mathbb{R}^d$. Our objective is to study an optimal domain for $\lambda ^s_{p,q}(\cdot )$ over $\mathcal{A}_{r,R}$, and also analyze the behavior of $\lambda ^s_{p,q}(\Omega _t)$ for $0\leq t <R-r$. In the local case $s=1$, Hersch~\cite{Hersch} considered the first eigenvalue $\lambda ^1_{2,2}(\cdot )$ of the Laplace operator over $\mathcal{A}_{r,R}$ with $d=2$, and proved that \emph{`$\lambda ^1_{2,2}(\Omega _t)$ attains its maximum only at $t=0$.'} Later, many authors proved this phenomenon for $d\geq 2$ by establishing \emph{`the strict monotonicity of $\lambda ^1_{p,p}(\Omega _t)$ of the $p$-Laplace operator.'} For example, Harrell-Kr\"oger-Kurata~\cite{HarrelKroger}, and independently Kesavan~\cite{Kesavan} for $p=q=2$; and Anoop-Bobkov-Sasi~\cite{Anoop18} for $1<p=q<\infty $. Further, Bobkov-Kolonitskii~\cite{BobKol20} proved the strict monotonicity of $\lambda ^1_{p,q}(\cdot )$ using the domain perturbations and polarization method for $p\in (1, \infty )$ and $q\in [1,p^*)$, where $p^*$ is the critical Sobolev exponent. In the nonlocal case $s\in (0,1)$, Djitte-Fall-Weth~\cite{Djitte-Fall-Weth2021} proved the strict monotonicity of $\lambda ^s_{p,q}(\Omega _t)$ when  $p=2$ and $q=1,2$. Indeed, the strict monotonicity implies that $\lambda ^s_{p,q}(\Omega _t)$ attains its maximum only at $t=0$. One of the main ingredients they have used is the Hadamard perturbation formula for the derivative of $\lambda _{p,q}^s(\cdot)$. The reflection methods and the strong comparison principle give a strict sign for this derivative. For $p=2$, the authors of~\cite{HarrelKroger} (when $s\in (0,1)$) and the authors of~\cite{Kesavan, Djitte-Fall-Weth2021} (when $s=1)$, determined a strict sign for the derivative of $\lambda ^1_{p,q}(\cdot)$ using the strong comparison principle. In the local case $s=1$, the strong comparison principle is unavailable for any $p\in (1, \infty )$ except $p=2$. However, the authors of~\cite{BobKol20, Anoop18} carefully utilized the geometry of annular domains and the qualitative properties of the first eigenfunctions to bypass the strong comparison principle. Establishing the Hadamard perturbation formula is challenging and highly depends on the structure of the differential operator involved. However, this formula for the derivative of $\lambda _{p,q}^s(\cdot)$ is established in~\cite[Proposition~1.1]{HarrelKroger} and in~\cite[Section 2{\&}3]{Kesavan} for $p=q=2$ and $s=1$; in~\cite[Theorem 1.3]{BobKol20} for $p\in (1, \infty )$, $q\in [1,p^*_s)$ and $s=1$; and in~\cite[Corollary 1.2]{Djitte-Fall-Weth2021} for $p=2$, $q=1,2$ and $s\in (0,1)$. To our knowledge, the Hadamard perturbation formula is not available for any $p, q$, and $s$ in their natural ranges.

Recently, Anoop-Ashok~\cite{Anoop-Ashok23} proved the strict monotonicity of the first eigenvalue of the $p$-Laplace operator for $\frac{2d+2}{d+2}<p<\infty $ via a rearrangement in $\mathbb{R}^d$ called~\emph{polarization}. This approach bypasses the usage of the Hadamard perturbation formula. One of the main tools in this approach is a strict Faber-Krahn type inequality involving polarization, which indicates that the first eigenvalue strictly decreases under polarization. The restriction on the exponent $p$ is required to get this strict Faber-Krahn type inequality by applying a general strong comparison principle due to Sciunzi~\cite[Theorem 1.4]{Sciunzi2014} for the $p$-Laplace operator. Another key idea is to express the translations of the inner ball as a rearrangement by polarization and then apply the strict Faber-Krahn inequality to get the strict monotonicity of the first eigenvalue. In this paper, we follow a similar approach to establish the strict monotonicity of $\lambda _{p,q}^s(\cdot )$ over $\mathcal{A}_{r, R}$ for any $p\in (1,\infty )$, $s\in (0,1),$ and $q \in [1, p_s^*)$. 

To state our results, we recall the notion of polarization for sets in $\mathbb{R}^d$, which was first introduced by Wolontis~\cite{Wolontis}. A \emph{polarizer} is an open affine-halfspace in $\mathbb{R}^d$. For a polarizer $H$ in $\mathbb{R}^d$, the polarization $P_H(\Omega )$ of $\Omega \subseteq \mathbb{R}^d$ is defined as
\begin{equation*}
    P_H(\Omega ) = \left[\left(\Omega \cup \sigma _H(\Omega )\right)\cap H \right] \cup \left[\Omega \cap \sigma _H(\Omega ) \right],
\end{equation*}
where $\sigma _H(\cdot )$ is the reflection in $\mathbb{R}^d$ with respect to the affine-hyperplane $\partial H$. It is easily verified that $P_H$ takes open sets to open sets, and $P_H$ is a rearrangement (i.e., it respects the set inclusion and preserves the measure) on $\mathbb{R}^d$. Now, we state our main results as a unified theorem.
\begin{theorem}\label{Thm-unified}
    Let $p \in (1, \infty),$ $s \in (0,1)$, and $q\in [1, p_s^*)$.
    \begin{enumerate}[label=\rm (\roman*)]
        \item  \textbf{Strict Faber-Krahn type inequality:} Let $H$ be a polarizer, and $\Omega $ be a bounded open set with $\mathcal{C}^{1,1}$-boundary in $\mathbb{R}^d$. Then $\la^s_{p,q}(\cdot )$ decreases under polarization. Further, let $\Omega \neq P_H(\Omega ) \neq \sigma _H(\Omega )$, then $\la^s_{p,q}(\cdot )$ strictly decreases under polarization.
        \item \textbf{Strict monotonicity:} Let $0 < r <R$. Then $\lambda _{p,q}^s\left(B_R(0)\setminus \overline{B}_r(t e_1)\right)$ is strictly decreasing for $0\leq t < R-r$. In particular, 
        \begin{equation*}
            \lambda ^s_{p,q}\left(B_R(0)\setminus \overline{B}_r(0)\right)=\max \limits _{0\leq t < R-r}\lambda ^s_{p,q}\left(B_R(0)\setminus \overline{B}_r(t e_1)\right).
        \end{equation*}
    \end{enumerate}
\end{theorem}
\begin{remark}
    (i) For $p=2$, $s\in (0,1)$ and $q=1, 2$, Theorem~\ref{Thm-unified}-(ii) provides an alternative proof for the strict monotonicity of $\lambda ^s_{2,q}\left(B_R(0)\setminus \overline{B}_r(t e_1)\right)$, for $0\leq t <R-r$, obtained by~\cite[Theorem~1.4]{Djitte-Fall-Weth2021}. Moreover, Theorem~\ref{Thm-unified}-(ii) extends their result for $q\in [1, 2^*_s)$.\\
    (ii) Our Theorem~\ref{Thm-unified}-(ii) is a nonlocal analog of the monotonicity results obtained by Bobkov-Kolonitsii~\cite{BobKol20} for the local problem $-\Delta _p u = f(u)$ with a particular non-linearity $f(u)=|u|^{q-2}u$, $q\in [1, p^*),$ where $p^*=\frac{dp}{d-p}$ when $d>p$ and $p^*=\infty $ when $d\leq p$.\\
    (iii) In view of~\eqref{relation} and Theorem~\ref{Thm-unified}-(ii), the fractional $p$-torsional rigidity $\mathscr{T}_{s,p}(\Omega _t)$ is strictly increasing for $0\leq t <R-r$.
\end{remark}

The strict Faber-Krahn type inequality gives the strict monotonicity of $\lambda _{p,q}^s\left(B_R(0)\setminus \overline{B}_r(t e_1)\right)$ for $0\leq t < R-r$, since we can express the translations of $\overline{B}_r(t e_1)$ in $B_R(0)$ as rearrangements by polarization (see Figure~\ref{fig:my_label1} and Proposition~\ref{Pol_Balls}-(iii)). We briefly describe our procedure to prove the strict Faber-Krahn type inequality in Theorem~\ref{Thm-unified}. The effect of polarization on the $L^q$-norm and the Gagliardo seminorm (see Proposition~\ref{norm properties}) proves the Faber-Krahn type inequality for $\lambda ^s_{p,q}(\cdot )$. To obtain the strict sign, we prove a version of the strong comparison principle involving a Sobolev function and its polarization (see Proposition~\ref{SCP}). When $\Omega \neq P_H(\Omega )\neq \sigma _H(\Omega )$ and equality holds in the Faber-Krahn type inequality, we prove a contradiction to the strong comparison principle in $\Omega \cap H$ with the help of two sets $A_H(\Omega )$ and $B_H(\Omega )$ (see Figure~\ref{fig:my_label} or Proposition~\ref{Propo-ball}-(v)).   

The remainder of the paper is organized as follows. In Section~\ref{Section 2}, we discuss some essential properties and P\'olya-Szeg\"o type inequality of polarization. In the same section, we also prove the regularity and strong comparison principle for the solutions of~\eqref{Eigen_eqn}. The proofs of our main results are given in Section~\ref{Section 3}. At the end of Section~\ref{Section 3}, we give some strict monotonicity results with respect to certain rotations and translations of a~\emph{`hole'} in general classes of domains.

\section{Polarization, regularity, and strong comparison principle}\label{Section 2}
In this section, we discuss polarization and some of its properties. Also, we prove some regularity results and a version of the strong comparison principle for the solutions of the fractional $p$-Laplace operator.
\subsection{Polarization and its properties} An open affine-halfspace in $\mathbb{R}^d$ is called a~\emph{polarizer}. The set of all polarizers in $\mathbb{R}^d$ is denoted by $\mathscr{H}$. We observe that, for any polarizer $H\in \mathscr{H}$ there exist $h\in \mathbb{S}^{d-1}$ and $a\in \mathbb{R}$ such that $H = \left\{x\in \mathbb{R}^d : x\cdot h <a \right\}.  $
For $H\in \mathscr{H}$, let $\sigma _H$ be the reflection in $\mathbb{R}^d$ with respect to $\partial H$. Then, for any $x\in \mathbb{R}^d$,
\begin{equation}\label{refl_formula}
    \sigma _H(x) = x -2(x\cdot h -a)h.
\end{equation}
The reflection of $A\subseteq \mathbb{R}^d$ with respect to $\partial H$ is $\sigma _H(A)=\{\sigma _H(x): x\in A\}$. It is straightforward to verify $\sigma _H(A^\mathsf{c})=\sigma _H(A)^\mathsf{c}$, $\sigma _H(A\cup B)=\sigma _H(A)\cup \sigma _H(B),$ and $\sigma _H(A\cap B)=\sigma _H(A)\cap \sigma _H(B)$  for any $A, B\subseteq \mathbb{R}^d$. Now, we define the polarization of sets and functions, see~\cite[Definition~1.1]{Anoop-Ashok23}.
\begin{definition}
Let $H\in \mathscr{H}$ and $\Omega \subseteq \mathbb{R}^d$. The polarization $P_H(\Omega )$ and the dual polarization $P^H(\Omega )$ of $\Omega $ with respect to $H$ are defined as
\begin{align*}
    P_H(\Omega ) &= \left[(\Omega \cup \sigma _H(\Omega ))\cap H\right] \cup [\Omega \cap \sigma _H(\Omega )],\\
    P^H(\Omega ) &= \left[(\Omega \cup \sigma _H(\Omega ))\cap H^\mathsf{c}\right] \cup [\Omega \cap \sigma _H(\Omega )].
\end{align*}
For $u:\mathbb{R}^d\longrightarrow \mathbb{R}$, the polarization $P_H(u):\mathbb{R}^d\longrightarrow \mathbb{R}$ with respect to $H$ is defined as
\begin{equation*}
    P_H u(x) = \left\{
    \begin{aligned}
        &\max\{u(x), u(\sigma _H(x))\}, \quad \text{ for } x \in H,\\
        &\min\{u(x), u(\sigma _H(x))\}, \quad \text{ for } x \in \mathbb{R}^d\setminus H.
    \end{aligned}
    \right.
\end{equation*}
For $u:\Omega \longrightarrow \mathbb{R}$, let $\widetilde{u}$ be the zero extension of $u$ to $\mathbb{R}^d$. The polarization $P_H u:P_H(\Omega ) \longrightarrow \mathbb{R}$ is defined as the restriction of $P_H \widetilde{u}$ to $P_H(\Omega )$.
\end{definition}
The polarization for functions on $\mathbb{R}^d$ is introduced by Ahlfors~\cite{Ahlfors73} (when $d=2$) and Baernstein-Taylor~\cite{Baernstein94} (when $d\geq 2)$. For further reading on polarizations and their applications, we refer the reader to~\cite{Anoop-Ash-Kesh, NirjanUjjalGhosh21, Brock04, Solynin12, Weth2010}. Now, we list some important properties of polarization.
\begin{figure}
    \centering
    \includegraphics[width=0.5\textwidth]{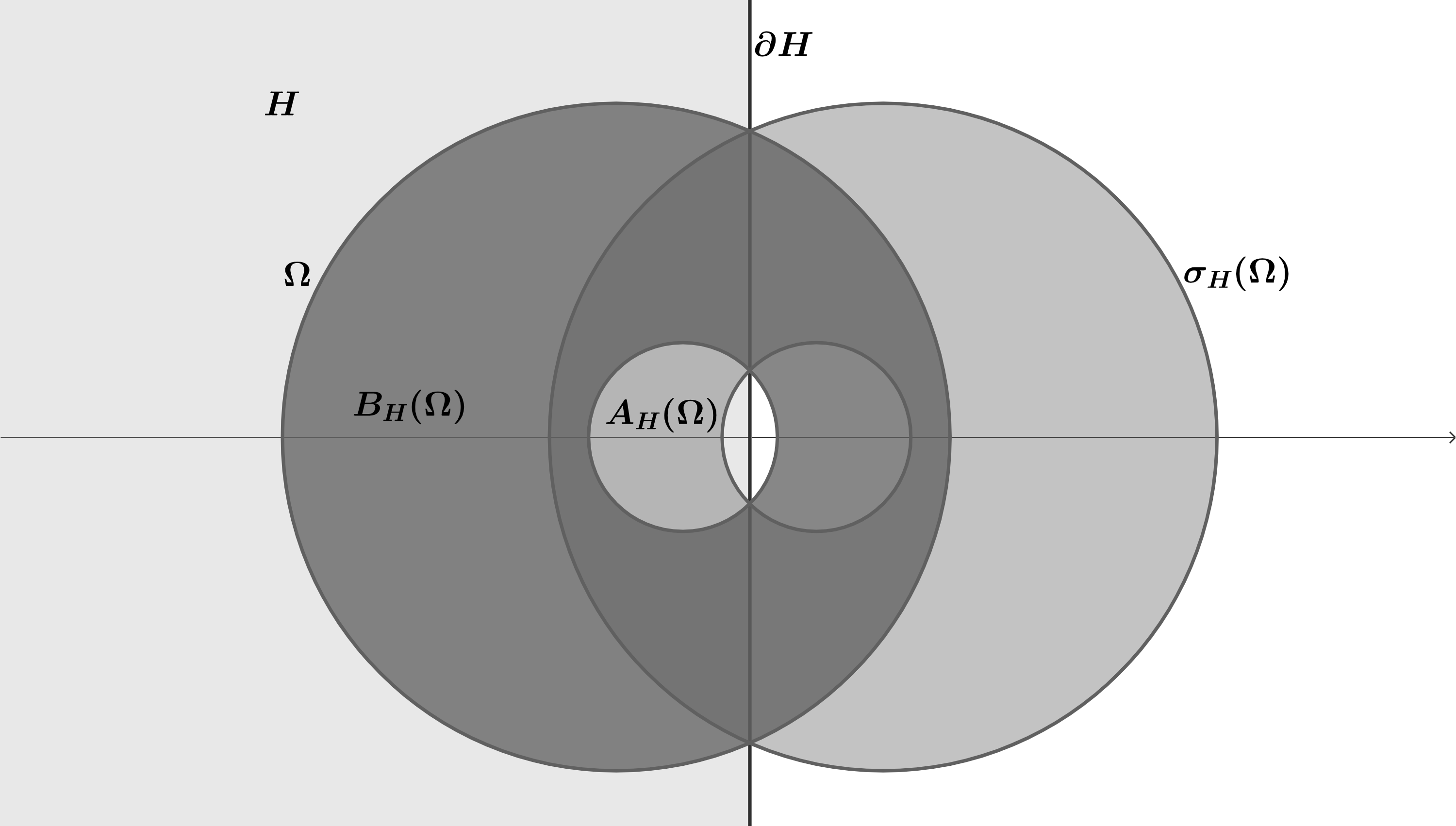}
    \caption{The sets $A_H(\Omega )$ and $B_H(\Omega )$ for an annular domain $\Omega$.}
    \label{fig:my_label}
\end{figure}
\begin{proposition}\label{Propo-ball}
    Let $H\in \mathscr{H}$, and $\Omega ,\mathscr{O} \subseteq \mathbb{R}^d.$ The following hold:
    \begin{enumerate}[label=\rm (\roman*)]
        \item $P_H(\sigma _H(\Omega ))=P_H(\Omega )$ and $P^H(\Omega ) = \sigma _H(P_H(\Omega ))$;
        \item $P_H(\Omega ^\mathsf{c})=(P^H(\Omega ))^\mathsf{c}$;
        \item $P_H(\Omega )=\Omega $ if and only if $\sigma _H(\Omega )\cap H\subseteq \Omega $;
        \item $P_H(\Omega \setminus \mathscr{O})=P_H(\Omega ) \setminus P^H(\mathscr{O}), \text{ if } \mathscr{O}, \sigma _H(\mathscr{O}) \subseteq \Omega $;
        \item if $\Omega $ is open and $\Omega \neq P_H(\Omega )\neq \sigma _H(\Omega )$, then the sets $A_H(\Omega ):= \sigma _H(\Omega ) \cap \Omega ^\mathsf{c} \cap H$ and $B_H(\Omega ):= \Omega \cap \sigma _H(\Omega ^\mathsf{c})\cap H $ have non-empty interiors.
    \end{enumerate}
\end{proposition}
\begin{proof}
    The polarizations of $\Omega $ can be rewritten as 
    \begin{align*}
        P_H(\Omega ) & = \left[(\Omega \cup \sigma _H(\Omega ))\cap H\right]\cup \left[\Omega \cap \sigma _H(\Omega )\cap H^\mathsf{c} \right],\\
        P^H(\Omega ) & = \left[(\Omega \cup \sigma _H(\Omega ))\cap \overline{H}^\mathsf{c}\right]\cup \left[\Omega \cap \sigma _H(\Omega )\right].
    \end{align*}
    (i) This follows from the fact that $\sigma _H (H)=\overline{H}^\mathsf{c}\in \mathscr{H}$.\\
    (ii) We have $P_H(\Omega^\mathsf{c})=\left[(\Omega^\mathsf{c}\cup \sigma _H(\Omega^\mathsf{c}))\cap H\right]\cup \left[\Omega^\mathsf{c}\cap \sigma _H(\Omega^\mathsf{c})\right]$. Therefore
    \begin{align*}
        \left(P_H\left(\Omega^\mathsf{c}\right)\right)^\mathsf{c} &=\left[(\Omega\cap \sigma _H(\Omega))\cup H^\mathsf{c}\right]\cap \left[\Omega\cup \sigma _H(\Omega)\right]\\
        &= \left[\Omega\cap \sigma _H(\Omega)\right]\cup \left[(\Omega\cup \sigma _H(\Omega))\cap H^\mathsf{c}\right]=P^H(\Omega).
    \end{align*}
   \noindent (iii) Suppose $P_H(\Omega )=\Omega $. Then the inclusion $\sigma _H(\Omega )\cap H\subseteq \Omega $ holds by noticing that $P_H(\Omega )\cap H = (\Omega \cap H) \cup (\sigma _H(\Omega )\cap H) =\Omega \cap H$. Now, assume that $\sigma _H(\Omega )\cap H\subseteq \Omega $. By applying $\sigma _H$ and using the facts that $\Omega \cap \partial H=\sigma _H(\Omega )\cap \partial H$ and $H^\mathsf{c}=\sigma _H(H)\cup \partial H$, we arrive at $\Omega \cap H^\mathsf{c}\subseteq \sigma _H(\Omega )$. Hence, $P_H(\Omega )\cap H = (\Omega \cap H)\cup (\sigma _H(\Omega )\cap H)=\Omega \cap H$, and $P_H(\Omega )\cap H^\mathsf{c}=\Omega \cap \sigma _H(\Omega )\cap H^\mathsf{c}=\Omega \cap H^\mathsf{c}$. Therefore $P_H(\Omega )=\Omega $.\\
    (iv) First, we observe that $(\Omega \cap \mathscr{O}^\mathsf{c})\cap \sigma _H(\Omega \cap \mathscr{O}^\mathsf{c}) = (\Omega \cap \sigma _H(\Omega ))\cap (\mathscr{O}^\mathsf{c}\cap \sigma _H(\mathscr{O}^\mathsf{c})).$ Also $\Omega \cup \sigma _H(\mathscr{O}^\mathsf{c})=\mathbb{R}^d$, since $\sigma _H(\mathscr{O})\subset\Omega$. Using this fact, we get $(\Omega \cap \mathscr{O}^\mathsf{c})\cup \sigma _H(\Omega \cap \mathscr{O}^\mathsf{c}) = (\Omega \cup \sigma _H(\Omega ))\cap (\mathscr{O}^\mathsf{c}\cup \sigma _H(\mathscr{O}^\mathsf{c})).$ Therefore, $P_H(\Omega \cap \mathscr{O}^\mathsf{c})=[((\Omega \cap \mathscr{O}^\mathsf{c})\cup \sigma _H(\Omega \cap \mathscr{O}^\mathsf{c}))\cap H]\cup [(\Omega \cap \mathscr{O}^\mathsf{c})\cap \sigma _H(\Omega \cap \mathscr{O}^\mathsf{c})]= P_H(\Omega )\cap P_H(\mathscr{O}^\mathsf{c}).$\\
    (v) Firstly, observe that the interiors of $A_H(\Omega )$ and $B_H(\Omega )$ are $\sigma _H(\Omega )\cap \overline{\Omega }^\mathsf{c}\cap H$ and $\Omega \cap \sigma _H(\overline{\Omega }^\mathsf{c})\cap H$ respectively. If $\sigma _H(\Omega )\cap \overline{\Omega }^\mathsf{c}\cap H = \emptyset $ then $\sigma _H(\Omega )\cap H\subseteq \Omega $, since both $\Omega $ and $H$ are open. Hence $P_H(\Omega )=\Omega $, from (iv). Similarly, using the fact that $B_H(\Omega )= A_H(\sigma _H(\Omega )) $, if $B_H(\Omega )$ has an empty interior, then $P_H(\sigma _H(\Omega ))=\sigma _H(\Omega )$.  This contradicts the hypothesis. Therefore, $A_H(\Omega )$ and $B_H(\Omega )$ have non-empty interiors.  
\end{proof}
In the following proposition, we prove the invariance of $L^q$-norm and a fractional P\'olya-Szeg\"o type inequality under polarization. The fractional P\'olya-Szeg\"o inequality is first established by Beckner~\cite[page 4818]{Beckner92}. For completeness, we give a proof motivated from~\cite[Section~3]{DeFernSalort21}.
\begin{proposition}\label{norm properties}
    Let $H\in \mathscr{H}$, $\Omega \subseteq \mathbb{R}^d$ be open, and let $u:\Omega \longrightarrow \mathbb{R}^+$ be measurable. If $u\in L^p(\Omega )$ for some $p\in [1,\infty )$, then $P_H(u)\in L^p(P_H(\Omega ))$ with $\left\|P_H(u)\right\|_{L^p(P_H(\Omega ))} = \left\|u\right\|_{L^p(\Omega )}$. Furthermore, if $u \in W^{s,p}_0(\Omega )$ for some $s\in (0,1)$, then $P_H(u) \in W^{s,p}_0(P_H(\Omega ))$ with $\left[P_H(u)\right]_{s,p} \leq \left[u\right]_{s,p}$.
\end{proposition}
\begin{proof}
    The inequality $\norm{P_H(u)}_{L^p(P_H(\Omega ))} =  \norm{u}_{L^p(\Omega )}$ follows from~\cite[Lemma~3.1-(i)]{Weth2010}. To prove the inequality $\left[P_H(u)\right]_{s,p} \leq \left[u\right]_{s,p}$, we split $\mathbb{R}^d \times \mathbb{R}^d$ as the union of the regions $H\times H$, $H^\mathsf{c}\times H^\mathsf{c}$, $H\times H^\mathsf{c}$, and $H^\mathsf{c}\times H$. Now, using the change of variables, we can rewrite $[u]_{s,p}^p$ as
    \begin{align*}
        \int _{\rd } \int _{\rd} \frac{|u(x)-u(y)|^p}{|x-y|^{d+sp}}   \dy \dx = \int _{H} \int _H & \left[\frac{|u(x)-u(y)|^p}{|x-y|^{d+sp}}\right.    + \frac{|u(\sigma_H(x))-u(y)|^p}{|\sigma _H(x)-y|^{d+sp}} \\
        & +  \frac{|u(x)-u(\sigma _H(y))|^p}{|x-\sigma _H(y)|^{d+sp}}    +\left. \frac{|u(\sigma _H(x))-u(\sigma _H(y))|^p}{|\sigma _H(x)-\sigma _H(y)|^{d+sp}} \right]  \dy \dx.
    \end{align*}
    A similar expression can be written for $v:=P_H(u)$. To prove the inequality $[P_H(u)]_{s,p}\leq [u]_{s,p}$, it is now enough to show
    \begin{multline}\label{eqn-PS1}
        \frac{|u(x)-u(y)|^p}{|x-y|^{d+sp}}+ \frac{|u(\sigma_H(x))-u(y)|^p}{|\sigma _H(x)-y|^{d+sp}} +  \frac{|u(x)-u(\sigma _H(y))|^p}{|x-\sigma _H(y)|^{d+sp}}+ \frac{|u(\sigma _H(x))-u(\sigma _H(y))|^p}{|\sigma _H(x)-\sigma _H(y)|^{d+sp}}\\
        \geq \frac{|v(x)-v(y)|^p}{|x-y|^{d+sp}}+ \frac{|v(\sigma_H(x))-v(y)|^p}{|\sigma _H(x)-y|^{d+sp}} +  \frac{|v(x)-v(\sigma _H(y))|^p}{|x-\sigma _H(y)|^{d+sp}}  + \frac{|v(\sigma _H(x))-v(\sigma _H(y))|^p}{|\sigma _H(x)-\sigma _H(y)|^{d+sp}},
    \end{multline}
    for all $x,y\in H$. Firstly, we notice that 
    \begin{equation}\label{eqn-Reflect1}
        |\sigma _H(x)-\sigma _H(y)|=|x-y|\leq |\sigma _H(x)-y|=|x-\sigma _H(y)|, ~~\text{ for } x,y\in H.
    \end{equation}
    We consider the following four different cases.\\
    \textbf{Case 1.} Let $x, y \in H$ be such that $u(x)\geq u(\sigma _H(x)), \text{ and } u(y)\geq u(\sigma _H(y)).$ Then, by the definition, $v(x)=u(x),$ $v(y)=u(y)$, $v(\sigma _H(x))=u(\sigma _H(x))$, and $v(\sigma _H(y))=u(\sigma _H(y))$. Therefore, \eqref {eqn-PS1} becomes equality.\\
    \textbf{Case 2.} Let $x, y \in H$ be such that $u(x)\leq u(\sigma _H(x)), \text{ and } u(y)\leq u(\sigma _H(y)).$ Then, by the definition, $v(x)=u(\sigma _H(x))$, $v(y)=u(\sigma _H(y))$, $v(\sigma _H(x))=u(x)$, and $v(\sigma _H(y))=u(y)$. Since the right-hand side of~\eqref{eqn-PS1} is invariant under $\sigma _H$, \eqref{eqn-PS1} becomes an equality in this case also.\\
    \textbf{Case 3.} Let $x, y \in H$ be such that $u(x)\leq u(\sigma _H(x)), \text{ and } u(y)\geq u(\sigma _H(y)).$ Then $v(x)=u(x)$, $v(\sigma _H(x))=u(\sigma _H(x))$, $v(y)= u(\sigma _H(y))$, and $v(\sigma _H(y))=u(y)$. Hence~\eqref{eqn-PS1} can be written as
    \begin{multline}
        \frac{|u(x)-u(y)|^p}{|x-y|^{d+sp}}+ \frac{|u(\sigma_H(x))-u(y)|^p}{|\sigma _H(x)-y|^{d+sp}} +  \frac{|u(x)-u(\sigma _H(y))|^p}{|x-\sigma _H(y)|^{d+sp}}+ \frac{|u(\sigma _H(x))-u(\sigma _H(y))|^p}{|\sigma _H(x)-\sigma _H(y)|^{d+sp}}\\
        \geq \frac{|u(x)-u(\sigma _H(y))|^p}{|x-y|^{d+sp}}+ \frac{|u(\sigma_H(x))-u(\sigma _H(y))|^p}{|\sigma _H(x)-y|^{d+sp}} + \frac{|u(x)-u(y)|^p}{|x-\sigma _H(y)|^{d+sp}}  + \frac{|u(\sigma _H(x))-u(y)|^p}{|\sigma _H(x)-\sigma _H(y)|^{d+sp}}.
    \end{multline}
    In view of~\eqref{eqn-Reflect1}, the above inequality holds provided
    \begin{equation}
       |u(x)-u(y)|^p-|u(\sigma _H(x))-u(y)|^p-|u(x)-u(\sigma _H(y))|^p+|u(\sigma _H(x))-u(\sigma _H(y))|^p \geq 0.
    \end{equation}
    This inequality is a direct consequence from~\cite[Lemma 1]{Chiti-Pucci79}.\\
    \textbf{Case 4.} Let $x, y \in H$ be such that $u(x)\geq u(\sigma _H(x)), \text{ and } u(y)\leq u(\sigma _H(y)).$ By interchanging the roles of $x$ and $y$, we arrive at Case 3; we get the inequality~\eqref{eqn-PS1}.\\
    Further, we have ${\rm supp}(P_H(u))\subseteq P_H(\supp{(u)})\subseteq P_H(\Omega )$ as $u$ is non-negative, see~\cite[Section~2]{Bob-Kol19} or~\cite[Proposition~2.14]{Anoop-Ashok23}. Hence, $P_H(u)=0$ in $\mathbb{R}^d\setminus P_H(\Omega )$. Therefore, we conclude that $P_H(u)\in W^{s,p}_0(P_H(\Omega )).$
\end{proof}
\subsection{Regularity results and strong comparison principle} 
In this subsection, we discuss the regularity of the $q$-eigenfunctions of~\eqref{Eigen_eqn} and the strong comparison principle for the fractional $p$-Laplace operator. For $q \ge p$, Franzina~\cite[Theorem 1.1]{Fr2019} showed that the minimizer of $\la_{p,q}^s$ defined over homogeneous fractional Sobolev space is bounded. In the following proposition, we provide an alternate proof of the global $L^{\infty}$-bound for the weak solutions of~\eqref{Eigen_eqn}.

\begin{proposition}\label{regular}
    Let  $p \in (1, \infty)$, $s \in (0,1)$, and $q \in [1,p^*_s)$. Let $\Omega \subset \R^d$ be a bounded open set, and $u \in \wps$ is a non-negative weak solution of~\eqref{Eigen_eqn}. Then the following hold: 
    \begin{enumerate}[label=\rm (\roman*)]
     \item $u \in L^{\infty}(\rd)$, i.e., $u$ is bounded on $\rd$.
     \item In addition, assume that $\Om$ is of class $\mathcal{C}^{1,1}$. Then $u \in \mathcal{C}^{0,\alpha }(\overline{\Omega })$ for some $\al \in (0, s]$. 
    \end{enumerate}
\end{proposition}
\begin{proof}
    (i) First, we show that $u \in L^{\infty}(\rd)$. We only consider $d>sp$ and $q \in (p, p^*_s)$. The proof follows similar arguments for $d \le sp$ and $q \in (p, \infty)$. For $q<p$, using~\cite[Theorem~4.1]{NB-SK22}, we get $u \in L^{\infty}(\rd)$. Set $M \geq 0$ and $\sigma \geq 1$. Define $u_M=\min\{u,M\}$ and $\phi=u_M^{\sigma}$. Clearly, $u_M, \phi \in L^{\infty}(\Omega ) \cap \wps.$ Choosing $\phi$ as a test function in the weak formulation of $u$, we arrive at
    \begin{align}\label{eqn:weak}
        \int _{\rd} \int _{\rd} \frac{|u(x)-u(y)|^{p-2}(u(x)-u(y))}{\abs{x-y}^{d+sp}} (\phi(x)-\phi(y))\, \dy \dx = \lambda \int_{\Omega } u(x)^{q-1} \phi(x) \, \dx. 
    \end{align}
    We use the following inequality given in~\cite[Lemma C.2]{BrLiPa}: 
    \begin{equation*}
        |a-b|^{p-2}(a-b)(a_M^\sigma -b_M^\sigma) \geq \frac{\sigma p^p}{(\sigma +p-1)^p}\left|a_M^{\frac{\sigma +p-1}{p}}-b_M^{\frac{\sigma +p-1}{p}}\right|^p,
    \end{equation*}
    where $a, b\geq 0$, $a_M=\min\{a,M\}$, and $b_M=\min\{b,M\}$; and the embedding $\wps \hookrightarrow  L^{p^*_s}(\rd)$ to estimate the left-hand side of the identity~\eqref{eqn:weak} as follows:
    \begin{align*}
        \int _{\rd} \int _{\rd} &\frac{|u(x)-u(y)|^{p-2}(u(x)-u(y))}{\abs{x-y}^{d+sp}} (\phi(x)-\phi(y))\, \dy \dx \\
        & \geq \frac{\sigma p^p}{(\sigma + p-1)^p} \int _{\rd} \int _{\rd}  \frac{\left|u_M(x)^{\frac{\sigma+p-1}{p}}-u_M(y)^{\frac{\sigma + p - 1}{p}}\right|^{p}}{\abs{x-y}^{d+sp}}  \, \dx \\ 
        & \geq\frac{C(d,s,p)\sigma p^p}{(\sigma+p-1)^p} \left(\int_{\rd}\left(u_M(x)^{\frac{\sigma+p-1}{p}} \right)^{p^*_{s}}\,\dx\right)^{\frac{p}{p^*_s}}. 
    \end{align*}
    Since $M$ is arbitrary, the monotone convergence theorem yields
    \begin{align}\label{bound1}
        \frac{C(d,s,p)\sigma p^p}{(\sigma+p-1)^p} 
        \left(\int_{\rd}\left(u(x)^{\frac{\sigma+p-1}{p}} \right)^{p^*_{s}}\,\dx\right)^{\frac{p}{p^*_s}} \le \lambda \int_{\Om} u(x)^{\sigma + q-1} \, \dx.
    \end{align}
    \noi \textbf{Step 1.} This step shows that $u^{\sigma_1 + p -1} \in L^{\frac{p^*_{s}}{p}}(\rd)$ for $\sigma_1 := p^*_{s} - p +1$. For $\sigma = \sigma_1$, \eqref{bound1} takes the following form: 
    \begin{align}\label{bound2}
        \frac{C(d,s,p)\sigma p^p}{(p^*_{s})^p} \left(\int_{\rd}u(x)^{\frac{p^*_{s}}{p}p^*_{s}} \,\dx\right)^{\frac{p}{p^*_{s}}} \le \la \int_{\Om} u(x)^{\sigma_1 + q-1} \, \dx.
    \end{align}
    We consider $A:= \{ x \in \Omega : u(x) \le R \}$ and $A^c=\Om\setminus A$ for $R>1$ to be chosen. For $q>p$, we write $\sigma_1 + q -1 = p^*_{s} + q-p$. Applying the H\"{o}lder's inequality with the conjugate pair $\left(\frac{p^*_s}{p}, \frac{p^*_s}{p^*_s - p}\right)$, the right-hand side of~\eqref{bound2} can be estimated as: 
    \begin{align}\label{bound3}
        \int_{\Om} u(x)^{\sigma_1 + q-1} \, \dx & = \left( \int_{A} + \int_{A^c} \right) u(x)^{\sigma_1 + q-1} \, \dx \no \\ & \le R^{\sigma_1 + q-1} \abs{A} + \left( \int_{\Om} u(x)^{\frac{p^*_{s}}{p}p^*_{s}} \, \dx \right)^{\frac{p}{p^*_{s}}} \left( \int_{A^c} u(x)^{\frac{(q-p)p^*_s}{p^*_s-p}} \, \dx  \right)^{\frac{p^*_s-p}{p^*_{s}}}.
    \end{align}
    Observe that $\frac{(q-p)p^*_s}{p^*_s-p} < p^*_s$ since $q<p^*_s$. Then, using the H\"{o}lder's inequality with conjugate pair $\left( \frac{p^*_s-p}{q-p}, \frac{p^*_s-p}{p^*_s-q} \right)$,
    \begin{align}\label{bound4}
        \int_{A^c} u(x)^{\frac{(q-p)p^*_s}{p^*_s-p}} \, \dx \le \abs{A^c}^{\frac{p^*_s-q}{p^*_s-p}} \left( \int_{A^c} u(x)^{p^*_s} \, \dx \right)^{\frac{q-p}{p^*_s - p}}.
    \end{align}
    Now, we choose $R$ sufficiently large so that 
    \begin{align*}
        \frac{(p^*_{s})^p}{C(d,s,p)\sigma p^p} \la \abs{A^c}^{\frac{p^*_s-q}{p^*_s}}  \left( \int_{\Om} u(x)^{p^*_s} \, \dx \right)^{\frac{q-p}{p^*_s}} < \frac{1}{2}.
    \end{align*}
    Therefore, from~\eqref{bound2},~\eqref{bound3}, and~\eqref{bound4} we obtain
    \begin{align*}
        \frac{1}{2} \left(\int_{\rd}u(x)^{\frac{p^*_{s}}{p}p^*_{s}} \,\dx\right)^{\frac{p}{p^*_{s}}}  \le \frac{(p^*_{s})^p}{C(d,s,p)\sigma p^p} \la R^{\sigma_1 + q-1} \abs{A}.
    \end{align*}
    Thus $u^{\sigma_1 + p -1} \in L^{\frac{p^*_{s}}{p}}(\rd)$.  
    
    \noi \textbf{Step 2.} This step contains the $L^{\infty}$-bound for $u$. Set $a= \frac{\sigma+p^*_s-1}{\sigma+q-1}$. Clearly $a>1$, since $q<p^*_s$. Applying Young's inequality with the conjugate pair $(a, a')$,
    \begin{align*}
        u(x)^{\sigma + q-1} \le \frac{u(x)^{\sigma+p^*_s-1}}{a} + \frac{1}{a'} \le u(x)^{\sigma + p^*_s-1} + 1.
    \end{align*}
    Moreover,  $\sigma +p -1 \le \sigma p$. Hence, from~\eqref{bound1} there exists $C=C(\la, d,s,p, \Omega )$ such that 
    \begin{align*}
        \left(\int_{\rd}u(x)^{\frac{\sigma+p-1}{p}p^*_{s}} \,\dx\right)^{\frac{p}{p^*_{s}}} \le  C \left( \frac{\sigma + p -1}{p}\right)^{p-1} \left(1 + \int_{\Om} u(x)^{\sigma + p^*_s-1} \, \dx \right).
    \end{align*}
    From the above inequality, we get 
    \begin{align*}
        \left(1+\int_{\rd}u(x)^{\frac{\sigma+p-1}{p}p^*_{s}} \,\dx\right)^{\frac{p}{p^*_{s}}} &\le 1 + \left(\int_{\rd}u(x)^{\frac{\sigma+p-1}{p}p^*_{s}} \,\dx\right)^{\frac{p}{p^*_{s}}} \\
        &\le 1 + C\left( \frac{\sigma + p -1}{p}\right)^{p-1} \left(1 + \int_{\Om} u(x)^{\sigma + p^*_s-1} \, \dx \right) \\
        &\le \left( 1 + C\left( \sigma + p -1 \right)^{p-1} \right) \left(1 + \int_{\Om} u(x)^{\sigma + p^*_s-1} \, \dx \right) \\
        &\le \widetilde{C} \left( \sigma + p -1 \right)^{p-1} \left(1 + \int_{\Om} u(x)^{\sigma + p^*_s-1} \, \dx \right),
    \end{align*}
    where $\widetilde{C} = \widetilde{C}(\la, d,s,p, \Omega ) = C + (p-1)^{1-p}$.
    Set $\vartheta=\sigma+p-1$. Then, the above inequality yields 
    \begin{align}\label{form 1}
        \left(1+ \int_{\rd} u(x)^{\frac{\vartheta}{p}p^*_{s}}\,\dx\right)^{\frac{p}{p^*_{s}(\vartheta - p)}} \le \widetilde{C}^{\frac{1}{\vartheta - p}} \vartheta^{\frac{p-1}{\vartheta - p}} \left(1+ \int_{\rd} u(x)^{p^*_{s}+ \vartheta - p} \, \dx \right)^{\frac{1}{\vartheta-p}}.
    \end{align}
    Now, we consider the following sequences $(\vartheta_n)$ defined as:
    \begin{align*}
        \vartheta_1 = p^*_{s}, \vartheta_2 = p + \frac{p^*_{s}}{p}(\vartheta_1 -p), \cdot \cdot \cdot,  \vartheta_{n+1} = p + \frac{p^*_{s}}{p}(\vartheta_n -p).   
    \end{align*}
    Observe that $p^*_{s}-p+ \vartheta_{n+1} = \frac{p^*_{s}}{p} \vartheta_{n}$, and $\vartheta_{n+1} = p + \left(\frac{p^*_{s}}{p}\right)^n(\vartheta_1 - p)$. Since $p^*_{s} >p$, $\vartheta_{n} \ra \infty$, as $n \ra \infty$. Further, from~\eqref{form 1},  
    \begin{align}\label{form 2}
        \left(1+ \int_{\rd} u(x)^{\frac{\vartheta_{n+1}}{p}p^*_{s}}\,\dx\right)^{\frac{p}{p^*_{s}(\vartheta_{n+1} - p)}} \le \widetilde{C}^{\frac{1}{\vartheta_{n+1} - p}} \vartheta_{n+1}^{\frac{p-1}{\vartheta_{n+1} - p}} \left(1+ \int_{\rd} u(x)^{\frac{p^*_{s}}{p}\vartheta_{n}} \, \dx \right)^{\frac{p}{p^*_{s}(\vartheta_{n}-p)}}.
    \end{align}
    Set $D_n := \left(1+ \int_{\rd} u(x)^{\frac{p^*_{s}}{p}\vartheta_{n}} \, \dx \right)^{\frac{p}{p^*_{s}(\vartheta_{n}-p)}}.$ We iterate~\eqref{form 2} up to $(n+1)$-th step to get 
    \begin{align}\label{form 3}
        D_{n+1} \le \displaystyle \widetilde{C}^{\sum_{k=2}^{n+1} \frac{1}{\vartheta_k - p}} \left( \prod_{k=2}^{n+1} \vartheta_{k}^{\frac{1}{\vartheta_{k} - p}}  \right)^{p-1} D_1.
    \end{align}
    Now $D_1 = \left(1+ \int_{\rd} u(x)^{\frac{p^*_{s}}{p}p^*_{s}} \, \dx \right)^{\frac{p}{p^*_{s}(p^*_{s}-p)}} \le C$ (by Step 1). Hence, from~\eqref{form 3}, we have 
    \begin{equation}\label{form 5}
        \norm{u}_{L^{\frac{p^*_{s} \vartheta_{n+1}}{p}}(\rd)}^{\frac{\vartheta_{n+1}}{\vartheta_{n+1}-p}} \leq  \widetilde{C}^{\sum_{k=2}^{n+1} \frac{1}{\vartheta_k - p}} \left( \prod_{k=2}^{n+1} \vartheta_{k}^{\frac{1}{\vartheta_{k} - p}}  \right)^{p-1} C.
    \end{equation}
    Moreover, 
    \begin{align*}
        \sum_{k=2}^\infty\frac{1}{\vartheta_k-p} = \frac{d}{s p(p^*_{s} - p)} \text{ and } \prod_{k=2}^{\infty}\vartheta_{k}^{\frac{1}{\vartheta_k - p}} \le C(d,s,p).
    \end{align*}
    Therefore, taking the limit as $n \ra \infty$ in~\eqref{form 5}, we conclude $u \in L^{\infty}(\rd)$. 
    
    \noi (ii) Since $u \in L^{\infty}(\rd)$ and $\Omega $ is of class $\mathcal{C}^{1,1}$, we apply~\cite[Theorem 1.1]{IaMoSq} to get $u \in \C^{0,\al}(\overline{\Om})$ for some $\al \in (0, s]$.
\end{proof}
\begin{remark}
  For any non-negative weak solution $u \in \wps$ to~\eqref{Eigen_eqn}, Proposition~\ref{regular} infers that $u \in L^{\infty}(\rd) \cap \C(\rd)$. 
\end{remark}
For $p\in (1,\infty )$, $s\in (0,1)$, and an open set $\Omega \subseteq \mathbb{R}^d$ the Sobolev space $\widetilde{W} ^{s,p}(\Omega )$ is defined as
\begin{equation*}
    \widetilde{W} ^{s,p}(\Omega )=\left\{u\in L^p_{loc}(\mathbb{R}^d): \int _{\Omega }\int _{\mathbb{R}^d} \frac{|u(x)-u(y)|^p}{|x-y|^{d+sp}}\dy \dx<\infty \right\}.
\end{equation*}
We remark that $W^{s,p}_0(\Omega )\subseteq \widetilde{W} ^{s,p}(\Omega ).$ Consider the following fractional equation
\begin{equation}\label{eqn-frac_harmonic}
    (-\Delta _p)^s u_1 - (-\Delta _p)^s u_2 \geq 0.
\end{equation}
We say a pair $(u_1, u_2)$ with $u_1, u_2 \in \widetilde{W} ^{s,p}(\Omega )$ solves~\eqref{eqn-frac_harmonic} weakly in $\Omega $, if for any $\phi \in \mathcal{C}_c^\infty (\Omega )$ with $\phi \geq 0$ the following inequality holds
\begin{multline*}
    \int _{\mathbb{R}^d} \int _{\mathbb{R}^d} \frac{|u_1(x)-u_1(y)|^{p-2}(u_1(x)-u_1(y))}{|x-y|^{d+sp}} (\phi (x) -\phi (y)) \, \dy\dx \\
    -\int _{\mathbb{R}^d}\int _{\mathbb{R}^d} \frac{|u_2(x)-u_2(y)|^{p-2}(u_2(x)-u_2(y))}{|x-y|^{d+sp}} (\phi (x) -\phi (y)) \, \dy\dx  \geq 0.
\end{multline*}
Since $\mathcal{C}_c^\infty (\Omega ')\subseteq \mathcal{C}_c^\infty (\Omega )$ for any $\Omega '\subseteq \Omega $, notice that the pair $(u_1,u_2)$ solves~\eqref{eqn-frac_harmonic} weakly in $\Omega '$ also if $(u_1,u_2)$ solves~\eqref{eqn-frac_harmonic} weakly in $\Omega $. In the literature, some authors obtained a version of the strong maximum principle for fractional $p$-Laplace operator on the upper half-space involving antisymmetric functions. For example, by Jarohs-Weth~\cite[Proposition~3.6]{JarohsWeth2016} for $p=2$; and by Chen-Li~\cite[Theorem~2]{CHCo2018} for $p\neq 2$ assuming $\mathcal{C}_{loc}^{1,1}$-regularity on the functions involved. Motivated by their results, we prove a variant of the strong comparison principle involving the antisymmetric function $w:= P_H(u) -u$  without $\mathcal{C}_{loc}^{1,1}$-regularity on $u \in W^{s,p}_0(\Omega )$ for $H \in \mathscr{H}$. For any set $A \subseteq \rd$, $|A|$ denotes the measure of $A$.
\begin{proposition}[Strong Comparison Principle]\label{SCP}
    Let $H\in \mathscr{H}$, $\Omega \subset \mathbb{R}^d$ be an open set, $p\in (1, \infty )$, $s \in (0,1)$, and $u\in W^{s,p}_0(\Omega )$ be non-negative. Assume that $P_H u$ and $u$ satisfy the following equation weakly:
    \begin{equation}\label{SCP-eqn}
        (-\Delta _p)^s P_H u - (-\Delta _p)^s u \geq 0 \mbox{ in } \Omega \cap H.
    \end{equation}
    Consider the following sets in $\Omega \cap H$:
    \begin{align*}
      \mathcal{A} := \left\{ x \in  \Omega \cap H : P_H u(x) = u(x)  \right\}; \; \mathcal{B} := \left\{ x \in  \Omega \cap H : P_H u(x) > u(x)   \right\}.
    \end{align*}
    If $|\mathcal{B}|>0$, then $\mathcal{A}$ has an empty interior.
\end{proposition}
\begin{proof}
    Firstly, we denote $v=P_Hu$ in $\mathbb{R}^d$ and $G(t)=|t|^{p-2}t$ with $G'(t)=(p-1)|t|^{p-2}\geq 0$ for $t\in \mathbb{R}$. On the contrary, we assume that the interior of $\mathcal{A}$ is nonempty. We consider a test function $\phi \in \mathcal{C}_{c}^\infty (\mathcal{A})$ with $\phi > 0$ on $K := \text{supp}(\phi)$ where $|K|>0$. From~\eqref{SCP-eqn} we have 
    \begin{equation}\label{scp-1}
        I:= \int _{\rd}\int _{\rd} \frac{G(v(x)-v(y))-G(u(x)-u(y))}{\abs{x-y}^{d+sp}} (\phi(x)-\phi(y))\, \dy \dx \geq 0.
    \end{equation}
    For brevity, we denote the reflection point $\sigma _H(z)$ by $\overline{z}$ for a point $z \in \rd$ and $H \in \mathscr{H}$. Using $\mathbb{R}^d=H\cup H^\mathsf{c}$ and $\overline{y}=\sigma _H(y)$, we write the inner integral of~\eqref{scp-1} as an integral over $H$ as 
    \begin{align*}
        \int _{\mathbb{R}^d} &\frac{G(v(x)-v(y))-G(u(x)-u(y))}{\abs{x-y}^{d+sp}} (\phi(x)-\phi(y))\, \dy\\
        &= \left[\int _{H}+\int _{H^\mathsf{c}}\right] \frac{G(v(x)-v(y))-G(u(x)-u(y))}{\abs{x-y}^{d+sp}} (\phi(x)-\phi(y))\, \dy\\
        &= \int _{H} \frac{G(v(x)-v(y))-G(u(x)-u(y))}{\abs{x-y}^{d+sp}} \phi(x)\, \dy -\int _{H} \frac{G(v(x)-v(y))-G(u(x)-u(y))}{\abs{x-y}^{d+sp}}\phi(y)\, \dy\\
        &~~~+ \int _{H} \frac{G(v(x)-v(\overline{y}))-G(u(x)-u(\overline{y}))}{\abs{x-\overline{y}}^{d+sp}} (\phi(x)-\phi(\overline{y}))\, \dy\\
        &=\int _{H} \left(\frac{1}{\abs{x-y}^{d+sp}}-\frac{1}{\abs{x-\overline{y}}^{d+sp}}\right) \big( G(v(x)-v(y))-G(u(x)-u(y)) \big) \phi(x)\, \dy\\
        &~~~-\int _{H} \frac{G(v(x)-v(y))-G(u(x)-u(y))}{\abs{x-y}^{d+sp}}\phi(y)\, \dy\\
        &~~~+\int _{H} \frac{G(v(x)-v(y))-G(u(x)-u(y))+G(v(x)-v(\overline{y}))-G(u(x)-u(\overline{y}))}{\abs{x-\overline{y}}^{d+sp}} \phi(x)\, \dy,
    \end{align*}
    where we have used $\phi (\overline{y})=0$ for $y\in H$. We express the integral $I=I_1-I_2+I_3$, where $I_1, I_2,$ and $I_3$ are given by
    \begin{align*}
        I_1&= \int_{H} \int _{H } \left(\frac{1}{\abs{x-y}^{d+sp}}-\frac{1}{\abs{x-\overline{y}}^{d+sp}}\right) \big( G(v(x)-v(y))-G(u(x)-u(y)) \big) \phi(x)\, \dy \dx,\\
        I_2&= \int _{\rd} \int _{H } \frac{G(v(x)-v(y))-G(u(x)-u(y))}{\abs{x-y}^{d+sp}}\phi(y)\, \dy \dx,\\
        I_3&= \int _H \int _{ H } \frac{G(v(x)-v(y))-G(u(x)-u(y))+G(v(x)-v(\overline{y}))-G(u(x)-u(\overline{y}))}{\abs{x-\overline{y}}^{d+sp}} \phi(x)\, \dy \dx,
    \end{align*}
    where in the integrals $I_1$ and $I_3$ we again used the fact that the support of $\phi$ lies inside $H$. Now, we show that $I_3=0$ and $I_2\geq 0$.
    Using $\mathbb{R}^d=H\cup H^\mathsf{c}$,  $\overline{x}=\sigma _H(x)$, and $v(y)=u(y)$ for $y\in K={\rm supp} (\phi )$, we rewrite $I_2$ as 
    \begin{multline}\label{I_2 first}
        I_2=\int _{H} \int _{K} \frac{G(v(x)-u(y))-G(u(x)-u(y))}{\abs{x-y}^{d+sp}}\phi(y)\, \dy \dx \\
        + \int _{H} \int _{K} \frac{G(v(\overline{x})-u(y))-G(u(\overline{x})-u(y))}{\abs{\overline{x}-y}^{d+sp}}\phi(y)\, \dy \dx.
    \end{multline}
    By the definition, either $v(x)=u(x)$ or $v(x)=u(\overline{x})$ for $x\in \mathbb{R}^d$. Consider the sets $K_1=\{x\in H : v(x)=u(\overline{x})\}$ and $K_2=\{x\in H : v(x)=u(x)\}$. Notice that both the terms in~\eqref{I_2 first} involving $G(\cdot)$ are zero on the set $K_2$. Hence, from~\eqref{I_2 first}, we get
    \begin{align*}
        I_2 = \int _{K_1}\int _K \Bigl(G(v(x)-u(y))-G(u(x)-u(y))\Bigr)\left(\frac{1}{|x-y|^{d+sp}}-\frac{1}{|\overline{x}-y|^{d+sp}}\right)\phi(y) \, \dy\dx.
    \end{align*}
    Observe that $|x-y| < |\overline{x} - y|$ for any $x,y \in H$, and  $\phi \geq 0$ in $H$. Moreover, since $G(\cdot)$ is increasing and $v(x) \ge u(x)$ for $ x \in H$, we get $G(v(x)-u(y))-G(u(x)-u(y)) \geq 0$ for $x\in H$ and $y\in K$. Thus, we conclude that $I_2\geq 0.$ 
    Similarly, by replacing $v(x)=u(x)$ on $K$ we rewrite $I_3$ as 
    \begin{align*}
        I_3 = \int_{K} \int _{H} \frac{G(u(x)-v(y))-G(u(x)-u(y))+G(u(x)-v(\overline{y}))-G(u(x)-u(\overline{y}))}{\abs{x-\overline{y}}^{d+sp}} \phi(x)\, \dy \dx.
    \end{align*}
    Now, the term in the numerator involving $G(\cdot )$ is zero for $x\in K$ and $y \in H=K_1 \cup K_2$. Therefore, we get $I_3=0$. Thus,~\eqref{scp-1} implies that 
    \begin{align}\label{inequality 1}
        I_1= I+I_2 \geq 0.
    \end{align}
     On the other hand, for a.e. $x \in K$ and $y \in H$, since $v(x)-u(x)=0$ for $x \in K$, we get  
    \begin{align*}
        (v(x)-v(y))-(u(x)-u(y)) = u(y)-v(y) \left\{\begin{array}{ll} 
             < 0, & \text {when } y \in \mathcal{B}; \\ 
             \le 0, & \text{when } y \in H \setminus \mathcal{B}. \\
             \end{array} \right.
    \end{align*}
    By the monotonicity of $G(\cdot )$:  for a.e. $x \in K$ and $y \in H$,
    \begin{align*}
        G(v(x)-v(y)) - G(u(x)-u(y)) \left\{\begin{array}{ll} 
             < 0, & \text {when } y \in \mathcal{B}; \\ 
             \le 0, & \text{when } y \in H \setminus \mathcal{B}. \\
             \end{array} \right.
    \end{align*}
    Further, we have $\phi > 0$ on $K$ and $|x-y| < |x-\overline{y}|$ for any $x,y \in H$. Therefore, 
    \begin{align*}
        &I_1 = \int _{K} \int _{H} \left(\frac{1}{\abs{x-y}^{d+sp}}-\frac{1}{\abs{x-\overline{y}}^{d+sp}}\right) \big[G(v(x)-v(y))-G(u(x)-u(y)) \big] \phi(x)\, \dy \dx \\
        &= \left[ \;\int _{K} \int _{\mathcal{B}} + \int _{K} \int _{H \setminus \mathcal{B} }\right] \left(\frac{1}{\abs{x-y}^{d+sp}}-\frac{1}{\abs{x-\overline{y}}^{d+sp}}\right) \big[G(v(x)-v(y))-G(u(x)-u(y)) \big] \phi(x)\, \dy \dx \\
        &:= I_{1,1} + I_{1,2} < 0,
    \end{align*}
where the last strict inequality follows from the facts that $I_{1,1} < 0$ (since $|\mathcal{B}|>0$ and $\abs{K}>0$), and $I_{1,2} \le 0$. Thus, $I_1 < 0$, a contradiction to \eqref{inequality 1}. Therefore, the set $\mathcal{A}$ can not have a nonempty interior. This concludes the proof.
\end{proof}
\section{Proofs for the main results}\label{Section 3} 
This section establishes a strict Faber-Krahn type inequality under polarization for $\lambda _{p,q}^s$. Consequently, we obtain the strict monotonicity of $\lambda _{p,q}^s$ over annular domains. Afterward, we state a remark about the strict monotonicity of $\lambda _{p,q}^s$ over other classes of Lipschitz domains. 

\begin{theorem}\label{Thm-FKtype}
Let $p\in (1,\infty )$, $s\in (0,1)$, and $q \in [1, p^*_s)$. Let $\Omega \subset \mathbb{R}^d$ be a bounded domain of class $\mathcal{C}^{1,1}$, and $H\in \mathscr{H}$ be a polarizer. Then
\begin{equation}\label{FKtype}
    \la^s_{p,q}\left(P_H(\Omega )\right)\leq \la^s_{p,q}(\Omega ).
\end{equation}
Further, if $\Omega \neq P_H(\Omega ) \neq \sigma _H(\Omega )$ then
\begin{equation}\label{StrictFKtype}
    \la^s_{p,q}\left(P_H(\Omega )\right)< \la^s_{p,q}(\Omega ).
\end{equation}
\end{theorem}
\begin{proof}
Let $u\in W^{s,p}_0(\Omega )$ be a non-negative $q$-eigenfunction associated with $\lambda _{p,q}^s(\Omega )$. By Proposition~\ref{norm properties} and the variational characterization of $\la^s_{p,q}$, we get that $P_H(u) \in W^{s,p}_0(P_H(\Omega ))$ with $\left\|P_H(u)\right\|_{L^q(P_H(\Omega ))} = \left\|u\right\|_{L^q(\Omega )} =1$, and  
    \begin{equation}\label{ineq-Thm~3.1}
        \lambda ^s_{p,q}(P_H(\Omega )) \leq \left[P_H(u)\right]_{s,p} \leq \left[u\right]_{s,p} = \lambda ^s_{p,q}(\Omega ) .
    \end{equation}
Now, we prove the strict inequality~\eqref{StrictFKtype}. From Proposition~\ref{regular}, $u \in \mathcal{C}(\overline{\Om})$ and hence  $u, P_H(u) \in \mathcal{C}(\rd)$.  We consider the following sets
\begin{align*}
    \mathcal{M}_u = \Bigl\{x\in P_H(\Omega )\cap H : P_H u(x) > u(x)\Bigr\} \text{ and } \mathcal{N}_u = \Bigl\{x\in P_H(\Omega )\cap H : P_H u(x) = u(x)\Bigr\}.
\end{align*}
Observe that $ \mathcal{M}_u$ is open and $\mathcal{N}_u$ is relatively closed in $P_H(\Omega )\cap H$. First, we find a set $B$ in the open set $\Omega \cap H$ such that $B\cap \mathcal{M}_u \neq \emptyset .$ 
Since $\Omega \neq P_H(\Omega )\neq \sigma _H(\Omega )$, from Proposition~\ref{Propo-ball}-(v) both $A_H(\Omega )=\sigma _H(\Omega ) \cap \Omega ^\mathsf{c} \cap H$ and $B_H(\Omega )=\Omega  \cap \sigma _H(\Omega ^\mathsf{c}) \cap H$ have non-empty interiors. Further, $u>0$ in $\Om$. By the definition, we have $P_H u\geq u$ in $P_H(\Omega )\cap H$, and 
    \begin{equation}\label{eq-3.1}
        \begin{aligned}
        \text{in } A_H(\Omega ): &~~~ u=0, ~u\circ \sigma _H>0, \text{ and hence } P_H u=u\circ \sigma _H>u;\\
        \text{in } B_H(\Omega ): &~~~ u>0, ~u\circ \sigma _H=0, \text{ and hence } P_H u=u.
    \end{aligned}
    \end{equation}
    Therefore, using~\eqref{eq-3.1}, $P_H(\Omega )\cap H = \mathcal{M}_u \sqcup \mathcal{N}_u$ with $\mathcal{M}_u\supseteq A_H(\Omega )$, and $\mathcal{N}_u\supseteq B_H(\Omega )$, here $A_H(\Omega ),$ $B_H(\Omega )$ have non-empty interiors. Also, by the definition of  $A_H(\Omega)$, we get $P_H(\Omega )\cap H =(\Omega \cap H)\sqcup A_H(\Omega)$. Thus, we have 
    \begin{align*}
        P_H(\Omega )\cap H = \mathcal{M}_u \sqcup \mathcal{N}_u = (\Omega \cap H)\sqcup A_H(\Omega) \text{ with } \mathcal{M}_u\supseteq A_H(\Omega ).
    \end{align*}
    Therefore, since $\mathcal{N}_u$ is relative closed in $P_H(\Omega )\cap H$ and $\Omega \cap H$ is an open set,
    we get $\mathcal{N}_u \subsetneq \Omega \cap H$, and hence the set $B:=\mathcal{M}_u \cap \Omega \cap H$ is a non-empty open set. Therefore, the sets $\mathcal{N}_u, B \subset \Omega \cap H$ have the following properties:
    \begin{equation}\label{eq-ball}
        P_H u > u \text{ in } B 
        \text{ and }
        P_H u \equiv u \text{ in } \mathcal{N}_u.
    \end{equation}
    On the contrary to~\eqref{StrictFKtype}, assume that $\lambda _{p,q}^s(P_H(\Omega )) = \lambda _{p,q}^s(\Omega ) = \lambda .$ Now, the equality holds in~\eqref{ineq-Thm~3.1}, and hence $P_H(u)$ becomes a non-negative minimizer of the following problem:
    \begin{align*}
    \lambda ^s_{p,q}(P_H(\Omega )) = \min_{v \in W_0^{s,p}(P_H(\Omega ))} \left\{ \left[v\right]_{s,p}^p : \left\|v\right\|_{L^q(P_H(\Omega ))} =1 \right\}.  
    \end{align*}
    As a consequence, the following equation holds weakly:
    \begin{equation}\label{eqn_1}
    (-\Delta _p)^s P_H(u) = \lambda^s_{p,q} (P_H(\Omega)) |P_H(u)|^{q-2}P_H(u) \mbox{ in } P_H(\Omega) , \qquad P_H(u)=0 \mbox{ in } \mathbb{R}^d\setminus P_H(\Omega).
    \end{equation}
    Since $\Omega \cap H \subsetneq P_H(\Omega ) \cap H$, both $u$ and $P_H u$ are weak solutions to the equation:
    \begin{align*}
        (-\Delta _p)^s \widetilde{w} -\lambda |\widetilde{w}|^{q-2} \widetilde{w} =0 \text{ in } \Omega \cap H,
    \end{align*}
    where $\la = \lambda^s_{p,q} (P_H(\Omega))$. Using $P_H u \geq u$ in $\Omega \cap H$, we see that the following equation holds weakly:
    \begin{align*}
        (-\Delta _p)^s P_H(u) - (-\Delta _p)^s u = \lambda (|P_H(u)|^{q-2}P_H(u) - |u|^{q-2} u) \geq 0 \text{ in } \Omega \cap H.
    \end{align*}
    Now $\mathcal{N}_u \subseteq \mathcal{A}$ with $\mathcal{N}_u$ having a non-empty interior. Moreover, since $B \subseteq \mathcal{B}$ is open, we have $|\mathcal{B}|>0$. Therefore, applying Proposition~\ref{SCP}, we get a contradiction to the statement~\eqref{eq-ball}. Thus, the strict inequality~\eqref{StrictFKtype} holds.
\end{proof}
 
For $0<r<R<\infty $ and $0\leq t < R-r$, recall that the annular domain $\Omega _t \in \mathcal{A}_{r,R}$ is given by
\begin{equation}\label{annuli}
    \Omega _t = B_R(0)\setminus \overline{B}_r(t e_1),
\end{equation}
where $e_1=(1,0,\cdots ,0)\in \mathbb{R}^d$, $B_r(a)=\{x\in \mathbb{R}^d : |x-a|<r\}$ is the open ball centered at $z\in \mathbb{R}^d$ with radius $r\geq 0$, and $\overline{B}_r(a)$ is the closure of $B_r(a)$ in $\mathbb{R}^d$. We are interested in studying the monotonicity of $\lambda _{p,q}^s(\Omega _t)$ for $0\leq t <R-r$. For this, we consider the following family of polarizers 
\begin{equation}\label{Polarizers}
    H_a := \Bigl\{x=(x_1,x')\in \mathbb{R}\times \mathbb{R}^{d-1} : x_1 < a\Bigr\} \text{ for } a\in \mathbb{R}.
\end{equation}
For $a\in \mathbb{R},$ we denote the reflection with respect to $\partial H_a$ by $\sigma _a = \sigma _{H_a}$, and the polarizations in $\mathbb{R}^d$ with respect to $H_a$ by $P_a = P_{H_a}$ and $P^a = P^{H_a}$. Recall, from~\eqref{refl_formula}, that the reflection of a point $x=(x_1,x')\in \mathbb{R}^d$ is given by $\sigma _a(x)=(2a-x_1, x')$. In the following proposition, we demonstrate the effect of polarization on balls and annular domains.
\begin{figure}
    \centering
    \includegraphics[width=0.5\textwidth]{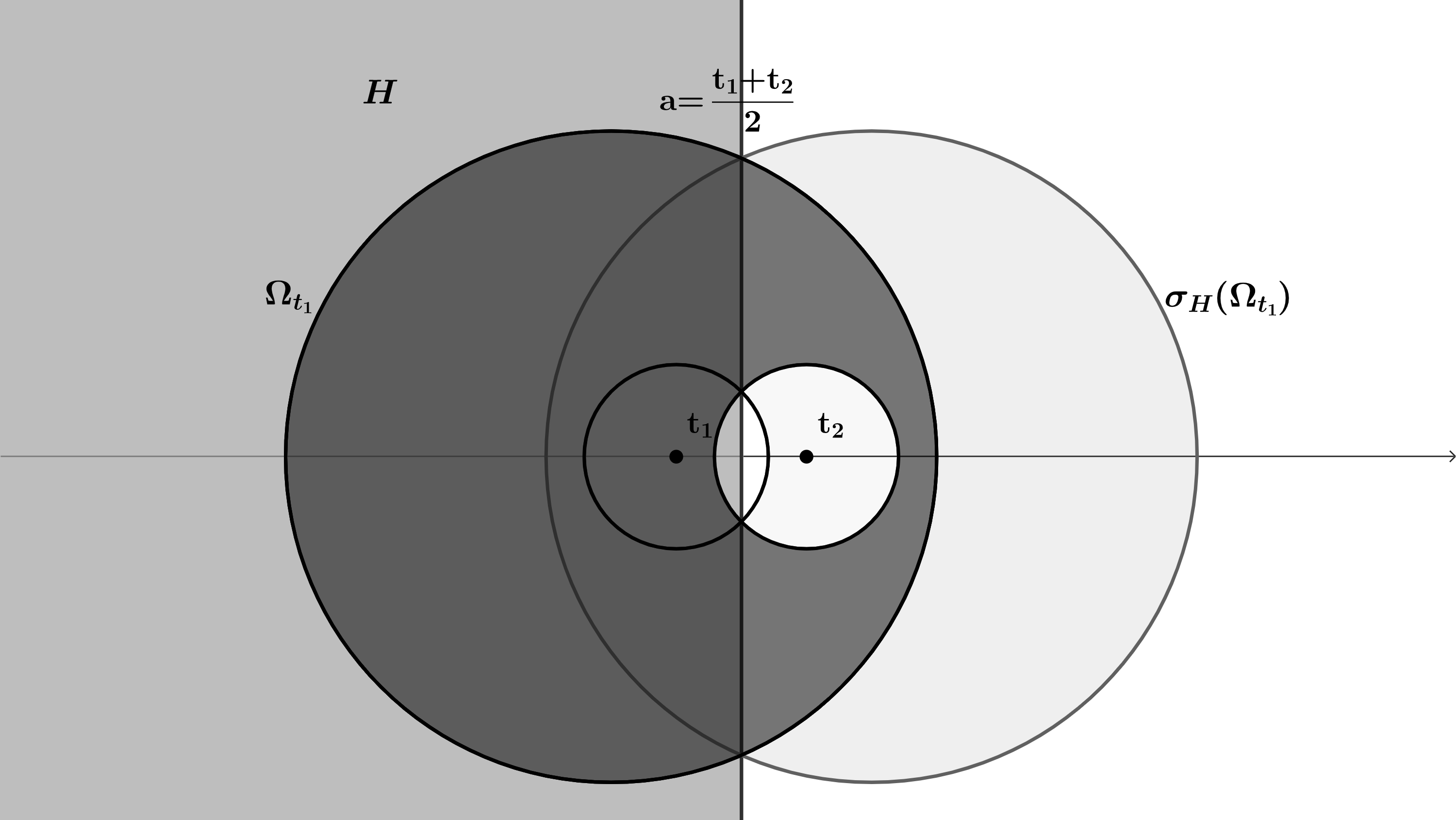}
    \caption{Polarization of $\Omega _{t_1}$ with respect to the polarizer $H_a=\{ x \in \rd : x_1 <a\}$. The dark-shaded region is $P_H(\Omega _{t_1})=\Omega _{t_2}$.}
    \label{fig:my_label1}
\end{figure}
\begin{proposition}\label{Pol_Balls}
    Let $0<r<R$, $0\leq t < R-r$, $a\in \mathbb{R}$, and $z=(z_1,z')\in \mathbb{R}^d$. Then:
    \begin{enumerate}[label=\rm (\roman*)]
        \item $P_a(B_r(z))=B_r(z)$ if $z_1\leq a$, and $P_a(B_r(z))=B_r(\sigma_a(z))$ if $z_1\geq a$;
        \item if $0\leq a <\frac{R-r+t}{2}$ then $\sigma _a (B_r(t e_1))\subset B_R(0)$;
        \item $P_a(\Omega _t)=\Omega _t$ for $0\leq a \leq t$; and $P_a(\Omega _t)=\Omega _{2a-t}$ for $t\leq a <\frac{R-r+t}{2}$.
    \end{enumerate}
\end{proposition}
\begin{proof}
    Recall that $P_a(B_r(z))=\left[ \left(B_r(z) \cup B_r(\sigma _a(z))\right)\cap H_a \right]\cup \left[B_r(z) \cap B_r(\sigma _a(z))\right]$.\\
    (i) To prove $P_a(B_r(z))=B_r(z)$ for $z_1 \leq a$, from Proposition~\ref{Propo-ball}-(iv), it is enough to show $B_r(\sigma _a(z))\cap H_a \subseteq B_r(z)$. For $x\in B_r(\sigma _a(z))\cap H_a$, we have $|\sigma _a (z)-x|<r$ and $x_1<a$. Since $x_1< 2a -x_1$ and $z_1\leq 2a -z_1$, we get $|z_1-x_1|<|(2a-z_1)-x_1|$. Then $|z-x|\leq |\sigma _a(z)-x|<r$, and thus $B_r(\sigma _a(z))\cap H_a \subseteq B_r(z)$. For $z_1\geq a$, we consider $P^a$ instead of $P_a$ to get $P^a(B_r(z))=B_r(z)$. Therefore, $P_a(B_r(z))=\sigma _a(B_r(z))=B_r(\sigma _a(z)).$\\
    (ii) Let $x\in \sigma _a(B_r(t e_1))=B_r((2a-t)e_1).$ Since $0\leq a <\frac{R-r+t}{2}$, we get $-(R-r)<-t \leq 2a -t <R-r$. Now $|x|\leq |x-(2a-t)e_1|+|(2a-t)e_1|<r+R-r=R$. Hence $x\in B_R(0).$\\
    (iii) Let $0\leq a <\frac{R-r+t}{2}$. From~(ii) we have $\sigma _a (B_r(t e_1))\subset B_R(0)$. Now, using Proposition~\ref{Propo-ball}-(v),
    \begin{equation}
        P_a(\Omega _t)= P_a\left(B_R(0)\setminus \overline{B}_r(t e_1)\right)=P_a\left(B_R(0)\right) \setminus P^a\left(\overline{B}_r(t e_1)\right). 
    \end{equation}
    From~(i), $P_a(B_R(0))=B_R(0)$, $P^a\left(\overline{B}_r(t e_1)\right)=\overline{B}_r(t e_1)$ if $a\leq t$; and $P^a\left(\overline{B}_r(t e_1)\right)=\overline{B}_r((2a-t)e_1)$ if $a\geq t$. Therefore $P_a(\Omega _t)=\Omega _t$ for $0\leq a \leq t$, and $P_a(\Omega _t)=B_R(0)\setminus \overline{B}_r((2a-t)e_1)=\Omega _{2a-t}$ for $t\leq a <\frac{R-r+t}{2}$.
\end{proof}

\begin{theorem}\label{thm-1.3} Let $p\in (1,\infty )$, $s \in (0,1)$, and $q\in [1, p_s^*)$. Then $\lambda _{p,q}^s(\Omega _{t_2}) < \lambda _{p,q}^s(\Omega _{t_1})$ for any $0\leq t_1 < t_2 < R-r$.
\end{theorem}
\begin{proof}
    Applying Proposition~\ref{Pol_Balls}-(iii) with $a=\frac{t_1+t_2}{2}$, we get $P_a(\Omega _{t_1})=\Omega _{t_2}$. Since $\Omega _{t_1}\neq \Omega _{t_2}\neq \sigma _a(\Omega _{t_1})$, from~\eqref{StrictFKtype} of Theorem~\ref{Thm-FKtype} we obtain that $\lambda _{p,q}^s(\Omega _{t_2}) = \lambda _{p,q}^s(P_a(\Omega _{t_1}))< \lambda _{p,q}^s(\Omega _{t_1})$.
\end{proof}

\subsection*{Some monotonicity results for general domains}
Our Theorem~\ref{Thm-FKtype} can be applied to a broader range of $\mathcal{C}^{1,1}$-smooth domains, as discussed in~\cite{Anoop-Ashok23}. Specifically, it can be used for domains of the form $\Omega \setminus \mathscr{O} \subset \mathbb{R}^d$, where $\mathscr{O}\subset \subset \Omega $ is a \emph{`hole.'} This application yields the monotonicity of $\lambda ^s_{p,q}(\Omega \setminus \mathscr{O})$ for certain rotations and translations of $\mathscr{O}$ within $\Omega $. In particular, for the class of domains that are the difference of Steiner symmetric or foliated Schwarz symmetric sets in $\mathbb{R}^d$. Examples of Steiner symmetric planar domains are regular $n$-gons and ellipses. Foliated Schwarz symmetric domains in $\mathbb{R}^2$ include rhombuses, disks, eccentric annular domains, and cones.

\noi \textbf{(a) Translations in a given direction.} We recall the following characterization of Steiner symmetric sets in $\mathbb{R}^d$ in terms of polarization, see for example~\cite[Lemma~2.2]{Bob-Kol19} or~\cite[Proposition~2.10-(i)]{Anoop-Ashok23}: \emph{A set $\Omega \subseteq \mathbb{R}^d$ is Steiner symmetric with respect to an affine-hyperplane $\partial H_{a_0}$, $a_0\in \mathbb{R}$ if and only if $P_a(\Omega )=\Omega $ for any $a\geq a_0$, and $P^a(\Omega )=\Omega $ for any $a\leq a_0$.} Assume that both $\Omega , \mathscr{O}$ are Steiner symmetric with respect to $\partial H_{0}$. The translations of $\mathscr{O}$ in the $e_1$-direction are given by $\mathscr{O}_t:= t e_1+ \mathscr{O}$ for $t\in \mathbb{R}$. Consider the set $L_\mathscr{O}$ of translations of $\mathscr{O}$ in the $e_1$-direction within $\Omega $ as 
\begin{equation*}
    L_\mathscr{O}:=\bigl\{t \geq 0 : \mathscr{O}_t\subset \Omega \bigr\}.
\end{equation*}
Then, from~\cite[Lemma~4.1 and proof of Theorem~1.5]{Anoop-Ashok23}, for $t_1<t_2$ in $L_\mathscr{O}$ we have $P^a(\mathscr{O}_{t_1})=\mathscr{O}_{t_2}$ with $2a =t_1+t_2$, in particular, $P_a\left(\Omega \setminus \mathscr{O}_{t_1}\right)=\Omega \setminus \mathscr{O}_{t_2}$; and the set $L_\mathscr{O}$ is an interval $[0, R_\mathscr{O})$, where $R_\mathscr{O}=\sup L_\mathscr{O}$. Now, the proofs Theorem~\ref{thm-1.3} and~\cite[Theorem~1.5]{Anoop-Ashok23} give the monotonicity of the eigenvalue $\lambda ^s_{p,q}(\Omega \setminus \mathscr{O}_t)$ for $0\leq t <R_\mathscr{O}$, and it is stated as the following theorem.
\begin{theorem}
    Let $\mathscr{O}, \Omega \subset \mathbb{R}^d$ be two sets such that $\Omega \setminus \mathscr{O}$ is $\mathcal{C}^{1,1}$-smooth bounded domain. Let $p\in (1,\infty )$, $s \in (0,1)$, and $q\in [1, p_s^*)$. If both $\mathscr{O}$ and $\Omega $ are Steiner symmetric with respect to $\partial H_0$, then $L_\mathscr{O}=[0, R_\mathscr{O})$ an interval, and $\lambda ^s_{p,q}(\Omega \setminus \mathscr{O}_t)$ is strictly decreasing for $0\leq t <R_\mathscr{O}$. 
\end{theorem}

\noindent \textbf{(b) Rotations about a given point.} We recall a characterization of foliated Schwarz symmetric sets in $\mathbb{R}^d$ in terms of polarizations, see for example~\cite[Proposition~2.10-(ii)]{Anoop-Ashok23}: \emph{A set $\Omega \subseteq \mathbb{R}^d$ is foliated Schwarz symmetric with respect to a ray $a+\mathbb{R}^+ \eta $, for some $a\in \mathbb{R}^d$ and $\eta \in\mathbb{S}^{d-1}$, if and only if $P_H(A)=A$ for any $H\in \mathscr{H}_{a, \eta }:=\left\{H\in\mathscr{H}: a+\mathbb{R}^+\eta \subset H \text{ and } a \in \partial H \right\}$.} Assume that both $\Omega , \mathscr{O}$ are foliated Schwarz symmetric with respect to the ray $a+\mathbb{R}^+\eta $, $a\in \mathbb{R}^d$, $\eta \in \mathbb{S}^{d-1}$. From~\cite[Proposition~2.9]{Anoop-Ashok23}, both $\mathscr{O}$ and $\Omega $ are axial-symmetric with respect to the straight line $a+\mathbb{R}\eta $. The rotations of $\mathscr{O}$ about the point $a\in \mathbb{R}^d$ are given by 
\begin{equation*}
    \mathscr{O}_{\theta , \xi }:= a+ R_{\theta , \xi }(-a+\mathscr{O}) \text{ for } \theta \in [0,\pi ],
\end{equation*}
where $\xi \in \mathbb{S}^{d-1}\setminus \{\eta \}$, and $R_{\theta , \xi }$ is the simple rotation in $\mathbb{R}^d$ with the plane of rotation $X_{\xi }:={\rm span}\left\{\eta , \xi \right\}$ and the angle of rotation $\theta $ from the ray $\mathbb{R}^+\eta $ in the counter-clockwise direction. By the axial-symmetry of both $\mathscr{O}$ and $\Omega $, and~\cite[Lemma~4.3]{Anoop-Ashok23}, it is enough to consider the rotations of $\mathscr{O}$ in $\Omega $ by $R_{\theta ,\xi }$ for a fixed $\xi \in \mathbb{S}^{d-1}\setminus \{\eta \}$. Therefore, we set $\mathscr{O}_{\theta }=\mathscr{O}_{\theta ,\xi }$ for $\theta \in [0,\pi ]$. Consider the set of rotations of $\mathscr{O}$ about the point $a\in \mathbb{R}^d$ within $\Omega $ as
\begin{equation*}
    C_{\mathscr{O}} := \bigl\{ \theta \in [0, \pi] : \mathscr{O}_\theta \subset \Omega \bigr\}.
\end{equation*}
Then, from~\cite[Lemma~4.4 and proof of Theorem~1.8]{Anoop-Ashok23}, for any $\theta _1< \theta _2$ in $C_\mathscr{O}$ there exists a polarization $H\in \mathscr{H}_{a, \eta }$ such that $P^H(\mathscr{O}_{\theta _2})=\mathscr{O}_{\theta _1}$, in particular $P_H\left(\Omega \setminus \mathscr{O}_{\theta _2}\right)=\Omega \setminus \mathscr{O}_{\theta _1}$; and the set $C_\mathscr{O}$ is an interval. Now, the proofs Theorem~\ref{thm-1.3} and~\cite[Theorem~1.8]{Anoop-Ashok23} give the monotonicity of  $\lambda ^s_{p,q}(\Omega \setminus \mathscr{O}_\theta )$ for $\theta \in C_\mathscr{O}$, and it is stated as the following theorem.
\begin{theorem}
    Let $\mathscr{O}, \Omega \subset \mathbb{R}^d$ be two sets such that $\Omega \setminus \mathscr{O}$ is $\mathcal{C}^{1,1}$-smooth bounded domain. Let $p\in (1,\infty )$, $s \in (0,1)$, and $q\in [1, p_s^*)$. If both $\mathscr{O}$ and $\Omega $ are foliated Schwarz symmetric with respect to the ray $a+\mathbb{R}^+\eta $ for some $a\in \mathbb{R}^d$ and $\eta \in \mathbb{S}^{d-1}$, then the set $C_\mathscr{O}$ is an interval, and  $\lambda ^s_{p,q}(\Omega \setminus \mathscr{O}_\theta )$ is strictly increasing for $\theta \in C_\mathscr{O}$.
\end{theorem}

\noi \textbf{Acknowledgments.}
We thank Prof. Vladimir Bobkov (Ufa FRC-RAS, Russia) and Prof. T.~V. Anoop (IIT Madras, India) for their valuable suggestions and comments, which improved the presentation of the article. We are grateful to the anonymous reviewers for their insightful comments, which have significantly enhanced the contents of our article. This work is partly funded by the Department of Atomic Energy, Government of India, under project no. 12-R$\&$D-TFR-5.01-0520. 

\end{document}